\documentclass{llncs}
\usepackage[pdftex]{graphicx}
\usepackage[pdftex]{color}
\usepackage{lmodern}
\usepackage{amssymb,bbold}
\usepackage{amssymb} 
\usepackage{mathtools}
\usepackage{amsfonts}
\usepackage[shortlabels]{enumitem}
\usepackage{natbib}
\usepackage{subcaption}
\newcommand{\iid}{\operatorname{\stackrel{i.i.d.}{\sim}}}

\newcommand{\cov}{\mbox{\rm Cov}}
\newcommand{\tr}{\mbox{\rm trace}}
\DeclareMathOperator*{\argmin}{\mbox{\rm argmin}}
\newtheorem{ConjN}{Conjecture}
\newtheorem{Def}{Definition}
\newtheorem{Th}[Def]{Theorem}

\newtheorem{As}[Def]{Assumptions}
\newtheorem{Lem}[Def]{Lemma}

\newcommand{\cD}{\mathcal{D}}
\newcommand{\cH}{\mathcal{H}}
\newcommand{\cN}{\mathcal{N}}

\newcommand{\Prb}{\mathbb P}

\newcommand{\diag}{\mbox{\rm diag}}

\newcommand{\inD}{\operatorname{\stackrel{\cD}{\to}}}

\newcommand{\mun}{\widehat{\mu}_n}
\newcommand{\xn}{\widehat{x}_n}

\newcommand{\cut}{\textnormal{Cut}}

\newcommand{\EE}{\mathbb{E}}
\newcommand{\NN}{\mathbb{N}}
\newcommand{\fm}{\mathfrak{m}}

\newcommand\KK{\textbf{K}}

\newcommand\dd{\textbf{d}}

\begin{document}
\title{Smeariness Begets Finite Sample Smeariness} 
\titlerunning{Finite Sample Smeariness} 
\author{Do Tran\inst{1}, Benjamin Eltzner\inst{1},  \and  Stephan Huckemann\inst{1}}
\authorrunning{Tran, Eltzner and Huckemann}
\institute{$^1$Georg-August-Universit\"at at G\"ottingen, Germany, Felix-Bernstein-Institute for Mathematical Statistics in the Biosciences, \\Acknowledging DFG HU 1575/7, DFG GK 2088, DFG SFB 1465 and the Niedersachsen Vorab of the Volkswagen Foundation}

\maketitle
\thispagestyle{plain}

\begin{abstract}

    Fr\'echet means are indispensable for nonparametric statistics on non-Euclidean spaces. For suitable random variables, in some sense, they ``sense'' topological and geometric structure. In particular, smeariness seems to indicate the presence of positive curvature. While smeariness may be considered more as an academical curiosity, occurring rarely, it has been recently demonstrated  that \emph{finite sample smeariness} (FSS) occurs  regularly on circles, tori and spheres and affects a large class of typical probability distributions. FSS can be well described by the \emph{modulation} measuring the quotient of rescaled expected sample mean variance and population variance. Under FSS it is larger than one -- that is its value on Euclidean spaces -- and this makes quantile based tests using tangent space approximations inapplicable. We show here that near smeary probability distributions there are always FSS probability distributions and as a first step towards the conjecture that all compact spaces feature smeary distributions, we establish directional smeariness under curvature bounds.      
\end{abstract}


\section{Introduction}

For nonparametric statistics of manifold data, the Fr\'echet mean plays a central role, both in descriptive and inferential statistics. For quite some while it was assumed that its asymptotics can be approximated under very general conditions by that of means of data projected to a suitable tangent space, e.g \cite{HL98,BP05,H_Procrustes_10,H_ziez_geod_10,BL17}. Under existence of second moments, these follow a classical central limit theorem. In the last decade, however, other asymptotic regimes have been discovered, yielding so called \emph{smeary} limiting rates, limiting rates that are slower than the classical $n^{-1/2}$, where $n$ denotes sample size, e.g. \cite{HH15,EltznereHuckemann2019}. While such smeary distributions are rather exceptional, more recently, it was discovered that these exceptional distributions affect the asymptotics of a large class of otherwise unsuspicious distributions, for instance all Fisher-von-Mises distributions on the circle, cf. \cite{HundrieserEltznerHuckemann2020}: for rather high sample sizes the rates are slower than  $n^{-1/2}$ and eventually an asymptotic variance can be reached that is higher than that of tangent space data. While this effect on the circle and the sphere is explored in more detail by \cite{EltznerHundrieserHuckemann2021}, here we concentrate on rather general manifolds and discuss recent findings concerning two conjectures.

\begin{ConjN}
~
\begin{itemize}
 \item[(a)] Whenever there is a random variable featuring smeariness, there are nearby random variables featuring finite sample smeariness.
 \item[(b)] All compact spaces feature smeariness.
\end{itemize}
\end{ConjN}

Here, we prove Conjecture (a) under the rather general concept of \emph{power smeariness} and Conjecture (b) for \emph{directional smeariness} under curvature bounds. We also provide for simulations, showing that classical quantile based tests fail under the presence of finite sample smeariness, suitably designed bootstrap tests, however, amend for it.

\section{Assumptions, Notation and Definitions}

Let $M$ be a complete Riemannian manifold of dimension $m \in \NN$ with induced distance $d$ on $M$. Random variables $X_1,\ldots,X_n \iid X$ on $M$ with silently underlying probability space $(\Omega,\Prb)$ induce \emph{Fr\'echet functions}
$$ F(p) = \EE[d(X,p)^2]\mbox{ and } F_n(p) = \frac{1}{n}\sum_{j=1}^n d(X_j,p)^2\mbox{ for }p\in M\,.$$
We also write $F^X$ and $F^X_n$ to refer to the underlying $X$. 

\begin{Lem}\label{lem:compact-mean-set}
 If $F(p) < \infty$ for some $p\in M$, the set of minimizers $\argmin_{p\in M} F(p)$ is not void and compact. In particular, $\argmin_{p\in M} F_n(p)$ admits a probability measure, uniform with respect to the Riemannian volume.
\end{Lem}

\begin{proof} If $F(p) < \infty$ for some $p\in M$, then due to the triangle inequality $F(p) < \infty$ for all $p\in M$. Further, by completeness of $M$ a minimizer of the Fr\'echet function is assumed, by continuity the set of minimizers is closed and due to
 $$ d(\mu_1,\mu_2) \leq \EE[d(\mu_1,X)] + \EE[d(\mu_2,X)] \leq \sqrt{\EE[d(\mu_1,X)^2]}+\sqrt{\EE[d(\mu_2,X)^2]}\,$$
it is bounded. Due to \cite{Na56}, $M$ can be isometrically embedded in a finite dimensional Euclidean space, hence the set of minimizers is compact. Thus the set of minimizers of $F(p)$, and as well those of $F_n(p)$ admit a probability measure, uniform with respect to the Riemannian volume. 
\end{proof}

We work under the following additional assumptions.

\begin{As}\label{As:1}
Assume
\begin{enumerate}
\item  $X$ is not a.s. a single point, 
 \item 
 $F(p)< \infty$ for some $p \in M$, 
 \item 
 there is a unique minimizer $\mu = \argmin_{p\in M} F(p)$, called the \emph{Fr\'echet population mean}, 
 \item 
 $\mun \in \argmin_{p\in M} F_n(p)$ is a selection from the set of minimizers uniform with respect to the Riemannian volume, called a \emph{Fr\'echet population mean},
\item and that the cut locus $\cut(\mu)$ of $\mu$ is either void or can be reached by two different geodesics from $\mu$.  
\end{enumerate}
The last point ensures that $\Prb\{X\in \cut(\mu)\} =0$ due to \cite{LeBarden2014}.
\end{As}

\begin{Def}
With the Riemannian exponential $\exp_\mu$, well defined on the tangent space $T_\mu M$, let 
\begin{eqnarray*}
 \rho(X,x) &:=& d(X,\exp_\mu x)^2\,,\\
 f(x) &:=& F(\exp_\mu(x))\,,\\
 f_n(x) &:=& F_n(\exp_\mu(x))\,.
\end{eqnarray*}
We also write $f^X$ and $f^X_n$ to refer to the underlying $X$. Further, with the Riemannian logarithm $\log_\mu p = (\exp_\mu)^{-1}(p)$, well defined outside of $\cut(\mu)$, we have
\begin{align}\label{eq:rho-expand}
  \rho(X,x) = \|\log_\mu X - x\|^2 + \mathcal{O}(|x|^2) \, .
\end{align}
We define 
	\begin{enumerate}
		\item[i.] the \emph{population variance} 
			$$V := F(\mu) = f(0) = \EE[d(X,\mu)^2] = \tr \left(\cov[\log_\mu X]\right)\,;$$
		\item[ii.] the \emph{Fr\'echet sample mean variance} 
				$$V_n :=  \EE[d(\mun,\mu)^2]; \text{ and }$$ 
		\item[iii.] the \emph{modulation}  
			\begin{eqnarray*}
 				\fm_n &:=& \frac{nV_n}{V}\,.
			\end{eqnarray*} 
	\end{enumerate}

We shall also write $\xn := \log_\mu \mun$ for the image of the empirical Fr\'echet mean $\hat{\mu}_n$ in the tangent space at $\mu$. Again, if necessary, we write $V^X, V_n^X, \mun^X, \xn^X$ and $\fm^X_n$ to refer to the underlying $X$.
\end{Def}

\begin{As}\label{As:2}
In order to reduce notational complexity, we also assume that 
\begin{eqnarray}\label{eq:Frecht-fcn}
 f(x) &=&  \sum_{j=1}^m \limits T_j |(R x)_j|^{r+2} + o (|x|^{r+2})
\end{eqnarray}
with some $r\geq 0$, where $(Rx)_j$ is the $j$-th component after multiplication with an orthogonal matrix $R$ and $T_1,\ldots,T_m$ are positive.
\end{As}

With these definitions, we can define various asymptotic regimes.

\begin{Def} \label{def:many_smeary}
 We say that $X$ is
 \begin{itemize}
  \item[(i)] \emph{Euclidean} if $\fm_n = 1$ for all $n \in \NN$,
  \item[(ii)] \emph{finite sample smeary} if $1 <\sup_{n\in\NN} \fm_n < \infty$,
  \item[(iii)] \emph{smeary} if $\sup_{n\in\NN} \fm_n = \infty$,
  \item[(iv)] \emph{$r$-power smeary} if (\ref{eq:Frecht-fcn}) holds with $r>0$.
 \end{itemize}
\end{Def}

If the manifold $M$ is a Euclidean space and if second moments of $X$ exist, then Assumptions \ref{As:1} hold and due to the classical central limit theorem, $X$ is then Euclidean (cf.~Definition~\ref{def:many_smeary}). In case of $M$ being a circle or a torus, as shown in \cite{HundrieserEltznerHuckemann2020}, $X$ is Euclidean only if it is sufficiently concentrated. As further shown there, if $X$ is spread beyond a geodesic half ball on the circle or the torus, it features finite sample smeariness, which, on the circle and the Torus manifests, among others, in two specific subtypes, cf. ({\bf contribution to this GSI2021}) 

In consequence of the \emph{general central limit theorem} (GCLT) from \cite{EltznereHuckemann2019} under Assumptions \ref{As:1} and \ref{As:2},
\begin{eqnarray}\label{eq:GCLT}
 n^{\frac{1}{2r+2}}(R^T\xn)_j &\inD& \cH_j\,\mbox{ for all }1\leq j \leq m
\end{eqnarray}
where $T= \diag(T_1,\ldots.T_m)$ and $(\cH_1 |\cH_1|^{r},\ldots,\cH_m |\cH_m|^{r})$ is multivariate Gaussian with zero mean and covariance
$$ \frac{4}{(r+2)^2} T^{-1}\,\cov[\log_\mu X] \,T^{-1}\,,$$
we have at once that $r$-power-smeary for $r>0$ implies smeariness.

\section{Smeariness begets Finite Sample Smeariness}
\begin{Th} 
 Under  Assumptions \ref{As:1} and \ref{As:2}, if there is a random variable on $M$ that is $r$-power smeary, $r>0$, then there is one that is finite sample smeary. More precisely, for every $K>1$ there is a random variable $Y$ with $\sup_{n\in\NN} \fm^{Y}_n \geq K$.
\end{Th}

\begin{proof}
 Suppose that $X$ is $r$-power smeary, $r>0$ on $M$. For given $K>0$ let $0<\kappa < 1$ such that $\kappa^{-2} = K$ and define the random variable $X_\kappa$ via
 $$ \Prb\{X_\kappa = \mu\} = \kappa\mbox{ and } \Prb\{X_\kappa = X \}= 1-\kappa\,.$$
   With the sets $A = \{X_\kappa = \mu\}$ and $B=\{X_\kappa = X \}$, the Fr\'echet function of $X_\kappa$ is given by
  \begin{eqnarray*}
    F^{X_\kappa}(p) &=& \int_A d(p,\mu)^2 \,d\Prb^{X_\kappa} + \int_B d(X,\mu)^2 \,d\Prb^{X_\kappa} \\
    &=& \kappa d(p,\mu)^2 + (1-\kappa) F^X(p)
  \end{eqnarray*}
  which yields that $X_\kappa$ has the unique mean $\mu$ and population variance 
  \begin{eqnarray}\label{proof:Th1}
  V^{X_\kappa} &=& F^{X_\kappa}(\mu) ~=~ (1-\kappa) F^X(\mu) ~=~ (1-\kappa) V^X\,.
  \end{eqnarray}
   by hypothesis. Since $\cov[\log_\mu X_\kappa] = (1-\kappa) \cov[\log_\mu X]$, we have thus with the GCLT (\ref{eq:GCLT}),
  $$ \sqrt{n} \xn^{X_\kappa} \inD \cN\left(0,\frac{1-\kappa}{\kappa^2}\,\cov[\log_\mu X]\right)\,.$$
  This yields $nV_n^{X_\kappa} = \frac{1-\kappa}{\kappa^2}V^X$ and in conjunction with (\ref{proof:Th1}) we obtain
  $$ \fm_n^{X_\kappa} = \frac{1}{\kappa^2}\,.$$
  Thus, $Y = X_{K^{-1/2}}$ has the asserted property. 
  \end{proof}

 \section{Directional Smeariness}

 \begin{Def}[Directional Smeariness]
  We say that $X$ is directional smeary if (\ref{eq:Frecht-fcn}) holds for $r=2$ with some of the $T_1,\ldots,T_m$ there equal to zero.
 \end{Def}

  \begin{Th} Suppose that $M$ is a Riemannian manifold with sectional curvature bounded from above by $\KK>0$ such that there exists a simply connected geodesic submanifold of constant sectional curvature $\KK$. Then $M$ features a random variable that is directional smeary.\end{Th}

\begin{proof} Let $\mu\in M$ and consider orthogonal unit vectors $V,W \in T_{\mu}M$ such that the sectional curvature along $\exp_{\mu}tW$ between $W$ and $V$ is $\KK$.
  Let us consider a point mass random variable $X$ with $P(\{X=\mu \})=1$, a geodesic $\gamma(t) = \exp(tV)$, and a family of random variables $X_t$ defined as
 		\[P\{X_t=\delta_{\gamma(t)} \} = P\{X_t=\delta_{\gamma(-t)} \} =\frac{1}{2}. \] 
 	We shall show that we can choose $t$ close to $\pi/\sqrt{\KK}$ and $\epsilon>0$ sufficiently small such that the random variable $Y_{t,\epsilon}$, which is defined as 
 		\[ P\{Y_{t,\epsilon} = X_t\}=\epsilon,\  P\{ Y_{t,\epsilon}=X\}=1-\epsilon, \]
 	is directional smeary.

 	Let us write $F^{t,\epsilon}(p)$ for the Fr\'echet function of $Y_{t,\epsilon}$. Suppose for the moment that $\mu$ is the unique Fr\'echet mean of $Y_{t,\epsilon}$, we shall show that for sufficient small $\epsilon$ and $t$ close to $\pi/\sqrt{K}$, the Hessian at $\mu$ of $F^{t,\epsilon}$ vanishes in some directions, which will imply that $Y_{t,\epsilon}$ is directional smeary as desired.
 
 We claim that $\nabla^2F^{t,\epsilon}(\mu)[W,W]=0,$ which will fulfill the proof. Indeed, it follows from \cite[Appendix B.2]{Tra21} that 
 		\[ \nabla^2F^{t,\epsilon}(\mu)[W,W]=(1-\epsilon)+2\epsilon(t\sqrt{\KK}) \cot (t\sqrt{\KK}). \] 
 	Hence, if we choose $t$ close to $\pi/\sqrt{\KK}$ such that 
 		\begin{equation} \label{eq:Th8:1}
			t\sqrt{\KK} \cot (t\sqrt{\KK}) = -\frac{1-\epsilon}{2\epsilon} 
		\end{equation} 
	then $ \nabla^2F^{t,\epsilon}(\mu)[W,W]=0$ as claimed.

  	It remains to show that $\mu$ is the unique Fr\'echet mean of $Y_{t,\epsilon}$ for $t$ close to $\pi/\sqrt{\KK}$ and $\epsilon$ satisfies Eq.\ref{eq:Th8:1}. Indeed, for any $p \in M$ we have 
 		\[F^{t,\epsilon}(p)=\frac{1}{2}\epsilon\dd^2(p,\gamma(t))+ \frac{1}{2}\epsilon\dd^2(p,\gamma(-t))+(1-\epsilon)\dd^2(p,\mu).\]
 	For small $\epsilon$ then $F^{t,\epsilon}(\mu) \leq F^{t,\epsilon}(p)$ if $\dd(p,\mu) \geq \pi/\sqrt{\KK}.$ Thus, it suffice to show that $\mu$ is the unique minimizer of the restriction of $F^{t,\epsilon}$ on the open ball $B(\mu,\pi/\sqrt{\KK})$. Let us consider a model in the two dimensional sphere $S^2_{\KK}$ of curvature $\KK$ with geodesic distance $\dd_s$. Let $\tilde{\mu}$ be the South Pole, $\tilde{V} \in T_{\tilde{\mu}}S^2_{\KK}$ be unit vector and $\tilde{\gamma}(t)=\exp_{\tilde{\mu}}t\tilde{V}$. Let $t$ and $\epsilon$ satisfy Eq.~\eqref{eq:Th8:1} and consider the following measure on $S^2_{\KK}$ 
 		\[ \zeta = (1-\epsilon)\delta_{\tilde{\mu}}+\frac{\epsilon}{2}(\delta_{\tilde{\gamma}(t)}+\delta_{\tilde{\gamma}(-t)}).\]
 	Write $F_{\zeta}$ for the Fr\'echet function of $\zeta$, then for any $q \in S^2_{\KK}$,
 		\[F_{\zeta}(q)=\frac{1}{2}\epsilon \dd_s^2(q,\tilde{\gamma}(t))+ \frac{1}{2}\epsilon\dd_s^2(q,\tilde{\gamma}(-t))+(1-\epsilon)\dd_s^2(q,\tilde{\mu}).\]
	It follows from the definition of $\zeta$ that $F_{\zeta}(\tilde{\mu})=F^{t,\epsilon}(\mu)$. Direct computation of $F_{\zeta}$ on the sphere $S^2_{\KK}$ verifies that $\tilde{\mu}$ is the unique Fr\'echet mean of $\zeta$. 
 	
 	On the other hand, suppose that $\tilde{p}\in S^2_{\KK}$ with $\dd_s(\tilde{p},\tilde{\mu})=\dd(p,\mu)<\pi/\sqrt{\KK}$ and $\angle(\log_{\tilde{\mu}}\tilde{p},\tilde{V})=\angle(\log_{\mu}p,V)$. Because $\KK$ is the maximum sectional curvature of $M$, Toponogov theorem, c.f. \cite[Theorem 2.2]{cheeger2008comparison} implies that 
 		 \[ \dd_s(\tilde{p},\tilde{\gamma}(-t))\leq \dd(p,\gamma(-t)) \text{ and } \dd_s(\tilde{p},\tilde{\gamma}(t))\leq \dd(p,\gamma(t)).\]
 		 Thus $F_{\zeta}(\tilde{p}) \leq F^{t,\epsilon}(p)$. Because $\tilde{\mu}$ is the unique minimizer of $F_{\zeta}$ and $F_{\zeta}(\tilde{\mu})=F^{t,\epsilon}(\mu)$ it follows that $\mu$ is the unique minimizer of $F^{t,\epsilon}|_{B(\mu,\pi/\sqrt{\KK})}$ as needed.
 	
\end{proof}

 \section{Simulations}

 In the analysis of biological cells' filament structures, buckles of \emph{microtubules} play an important role, e.g. \cite{nolting2014mechanics}. For illustration of the effect of FSS in Kendall's shape spaces $\Sigma_m^k$ of $k$ landmark configurations in the $m$-dimensional Euclidean space, e.g. \cite{DM16}, we have simulated two groups of $20$ planar buckle structures each without and with the presence of \emph{intermediate vimentin} filaments (generating stiffness) and placed 5 mathematically defined landmarks on them, leading to two groups in $\Sigma_2^5$ as detailed in \cite{tran2021reverse}. Figure \ref{fig:shapes} shows typical buckle structures. 
 
 We compare the two-sample test based on suitable $\chi^2$-quantiles in tangent space with the test based on a suitable bootstrap procedure amending for FSS, cf. \cite{HundrieserEltznerHuckemann2020, EltznerHundrieserHuckemann2021}. In order to assess the effective level of the test we have generated a control sample of another $20$ buckles in the presence of vimentin filaments. As clearly visible in Table \ref{table:test}, the presence of FSS results in an higher level size of the quantile-based test and a reduced power, thus making it useless for further evaluation. In contrast the bootstrap-based test keeps the level, making its rejection of equality of buckling with and without vimentin credible.
 
\begin{figure}[h!]
\centering
\includegraphics[width=0.9\textwidth]{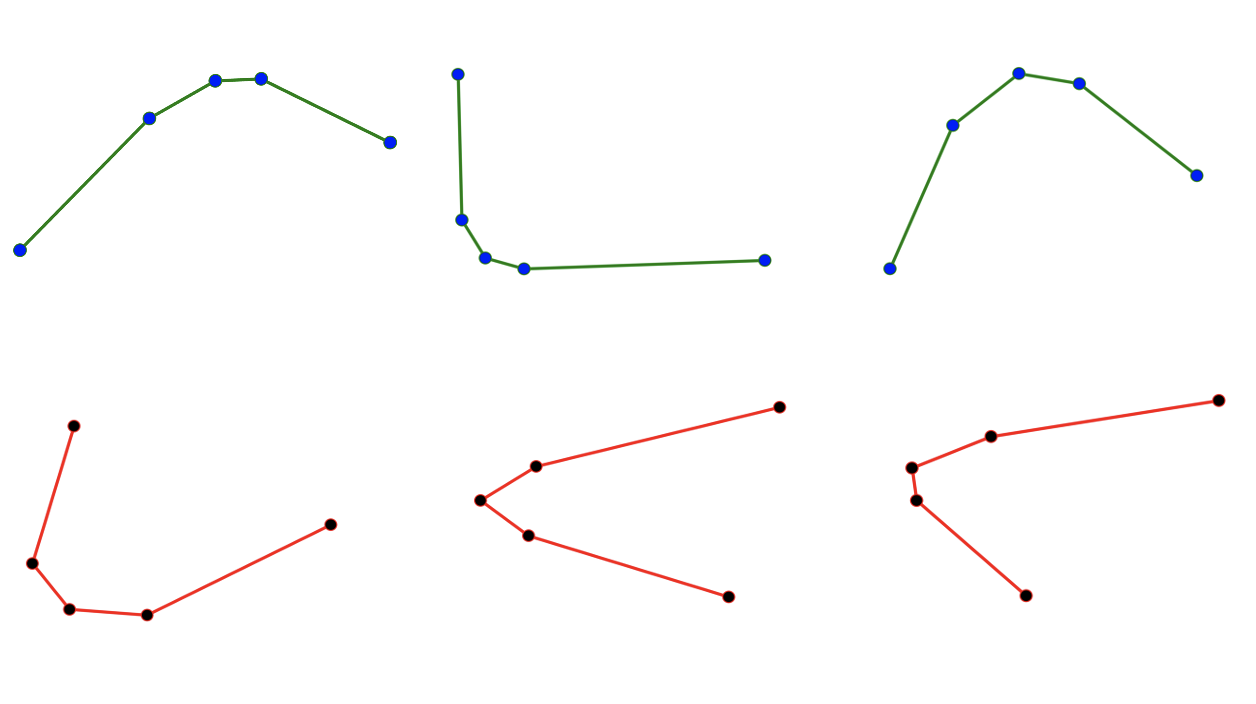}
\caption{\it Microtubule buckle structures with 5 landmarks. Upper row: without vimentin filaments. Lower row: in the presence of vimentin filaments (generating stiffness).
}
\label{fig:shapes}
\end{figure}
\begin{table}[ht]
\centering	
\caption{\it Reporting fraction of rejected hypothesis of equality of means using 100 simulations of two-sample test at nominal level $\alpha = 0.05$ based on quantiles (top row) and a suitable bootstrap procedure (bottom row) under equality (left column) and under inequality (right column).}

\vspace*{0.5cm}
\begin{tabular}{ c||c|c|c| } 
 \hline
 &Both with vimentin  & One with and the other without vimentin  \\ 
  \hline
 \hline
Quantile based& 0.11 & 0.37 \\ 
 \hline
Bootstrap based & 0.03 & 0.84  \\ 
 \hline

\end{tabular}

\label{table:test}
\end{table}



\begin{thebibliography}{}
  \thispagestyle{plain}

\bibitem[\protect\citeauthoryear{Bhattacharya and Lin}{Bhattacharya and
  Lin}{2017}]{BL17}
Bhattacharya, R. and L.~Lin (2017).
\newblock Omnibus {CLT}s for {F}r{\'e}chet means and nonparametric inference on
  non-{E}uclidean spaces.
\newblock {\em Proceedings of the American Mathematical Society\/}~{\em
  145\/}(1), 413--428.

\bibitem[\protect\citeauthoryear{Bhattacharya and Patrangenaru}{Bhattacharya
  and Patrangenaru}{2005}]{BP05}
Bhattacharya, R.~N. and V.~Patrangenaru (2005).
\newblock Large sample theory of intrinsic and extrinsic sample means on
  manifolds {II}.
\newblock {\em The Annals of Statistics\/}~{\em 33\/}(3), 1225--1259.

\bibitem[\protect\citeauthoryear{Cheeger and Ebin}{Cheeger and
  Ebin}{2008}]{cheeger2008comparison}
Cheeger, J. and D.~G. Ebin (2008).
\newblock {\em Comparison theorems in Riemannian geometry}, Volume 365.
\newblock American Mathematical Soc.

\bibitem[\protect\citeauthoryear{Dryden and Mardia}{Dryden and
  Mardia}{2016}]{DM16}
Dryden, I.~L. and K.~V. Mardia (2016).
\newblock {\em Statistical Shape Analysis\/} (2nd ed.).
\newblock Chichester: Wiley.

\bibitem[\protect\citeauthoryear{Eltzner and Huckemann}{Eltzner and
  Huckemann}{2019}]{EltznereHuckemann2019}
Eltzner, B. and S.~F. Huckemann (2019).
\newblock A smeary central limit theorem for manifolds with application to
  high-dimensional spheres.
\newblock {\em Ann. Statist.\/}~{\em 47\/}(6), 3360--3381.

\bibitem[\protect\citeauthoryear{Eltzner, Hundrieser, and Huckemann}{Eltzner
  et~al.}{2021}]{EltznerHundrieserHuckemann2021}
Eltzner, B., S.~Hundrieser, and S.~F. Huckemann (2021).
\newblock Finite sample smeariness on spheres.

\bibitem[\protect\citeauthoryear{Hendriks and Landsman}{Hendriks and
  Landsman}{1998}]{HL98}
Hendriks, H. and Z.~Landsman (1998).
\newblock Mean location and sample mean location on manifolds: asymptotics,
  tests, confidence regions.
\newblock {\em Journal of Multivariate Analysis\/}~{\em 67}, 227--243.

\bibitem[\protect\citeauthoryear{Hotz and Huckemann}{Hotz and
  Huckemann}{2015}]{HH15}
Hotz, T. and S.~Huckemann (2015).
\newblock Intrinsic means on the circle: Uniqueness, locus and asymptotics.
\newblock {\em Annals of the Institute of Statistical Mathematics\/}~{\em
  67\/}(1), 177--193.

\bibitem[\protect\citeauthoryear{Huckemann}{Huckemann}{2011a}]{H_Procrustes_10}
Huckemann, S. (2011a).
\newblock Inference on 3{D} {P}rocrustes means: Tree boles growth,
  rank-deficient diffusion tensors and perturbation models.
\newblock {\em Scandinavian Journal of Statistics\/}~{\em 38\/}(3), 424--446.

\bibitem[\protect\citeauthoryear{Huckemann}{Huckemann}{2011b}]{H_ziez_geod_10}
Huckemann, S. (2011b).
\newblock Intrinsic inference on the mean geodesic of planar shapes and tree
  discrimination by leaf growth.
\newblock {\em The Annals of Statistics\/}~{\em 39\/}(2), 1098--1124.

\bibitem[\protect\citeauthoryear{Hundrieser, Eltzner, and Huckemann}{Hundrieser
  et~al.}{2020}]{HundrieserEltznerHuckemann2020}
Hundrieser, S., B.~Eltzner, and S.~F. Huckemann (2020).
\newblock Finite sample smeariness of {F}r\'echet means and application to
  climate.

\bibitem[\protect\citeauthoryear{Le and Barden}{Le and
  Barden}{2014}]{LeBarden2014}
Le, H. and D.~Barden (2014).
\newblock On the measure of the cut locus of a {F}r{\'e}chet mean.
\newblock {\em Bulletin of the London Mathematical Society\/}~{\em 46\/}(4),
  698--708.

\bibitem[\protect\citeauthoryear{Nash}{Nash}{1956}]{Na56}
Nash, J. (1956).
\newblock The imbedding problem for {R}iemannian manifolds.
\newblock {\em The Annals of Mathematics\/}~{\em 63}, 20--63.

\bibitem[\protect\citeauthoryear{Nolting, M{\"o}bius, and K{\"o}ster}{Nolting
  et~al.}{2014}]{nolting2014mechanics}
Nolting, J.-F., W.~M{\"o}bius, and S.~K{\"o}ster (2014).
\newblock Mechanics of individual keratin bundles in living cells.
\newblock {\em Biophysical journal\/}~{\em 107\/}(11), 2693--2699.

\bibitem[\protect\citeauthoryear{Tran}{Tran}{2019}]{Tra21}
Tran, D. (2019).
\newblock Behavior of {F}r\'echet mean and central limit theorems on spheres.
\newblock {\em Brasilian Journal of Probability and Statistics, arXiv preprint
  arXiv:1911.01985\/}.
\newblock to appear.

\bibitem[\protect\citeauthoryear{Tran, Eltzner, and Huckemann}{Tran
  et~al.}{2021}]{tran2021reverse}
Tran, D., B.~Eltzner, and S.~F. Huckemann (2021).
\newblock Reflection and reverse labeling shape spaces and analysis of
  microtubules buckling.
\newblock {\em manuscript\/}.

\end{thebibliography}

\end{document}